\documentclass[a4paper,11pt]{amsart}

\usepackage{amsmath,amsfonts,amssymb,amsthm,amscd}
\usepackage{mathtools} 
\usepackage{latexsym} 
\usepackage[english]{babel}
\usepackage{tikz}
\usepackage{tikz-cd}
\usepackage[all]{xy}
\usepackage{lineno}

\theoremstyle{definition}

\newtheorem{remark}[equation]{Remark}
\newtheorem{example}[equation]{Example}

\theoremstyle{plain}

\newtheorem{theorem}[equation]{Theorem}
\newtheorem{proposition}[equation]{Proposition}
\newtheorem{corollary}[equation]{Corollary}

\newtheorem{question}[equation]{Question}

\numberwithin{equation}{section}
\numberwithin{figure}{section}
\numberwithin{table}{section}

\newcommand{\C}{\mathbb C} 

\newcommand{\Q}{\mathbb Q} 
\newcommand{\R}{\mathbb R} 

\newcommand{\Z}{\mathbb Z} 

\renewcommand{\tilde}{\widetilde}
\renewcommand{\bar}{\overline}

\newcommand{\supp}{\operatorname{supp}}
\renewcommand{\div}{\ensuremath{\operatorname{div}}}

\newcommand{\Cone}{\ensuremath{\operatorname{Cone}}}

\DeclareMathOperator{\wt}{wt}

\newcommand{\QED}{\hfill\end{proof}}
\newcommand{\erem}{~\hfill$\lozenge$\end{remark}\vskip 3pt}
\newcommand{\eex}{~\hfill$\lozenge$\end{example}\vskip 3pt}

\title[Hodge ideal and spectrum]{Hodge ideal and spectrum of weighted homogeneous isolated singularities}
\author{Seung-Jo Jung, In-Kyun Kim, Youngho Yoon}

\address{
Department of Mathematical Sciences, 
Seoul National University, 
GwanAkRo 1,  Gwanak-Gu,  
Seoul 08826, Korea
}
\email{toyul419@snu.ac.kr}
\email{soulcraw@snu.ac.kr}
\email{nsyyh@snu.ac.kr}

\begin{document}
\thanks{This work was supported by BK21 PLUS SNU Mathematical Sciences Division and the National Research Foundation of Korea(NRF) grant funded by the Ministry of  Science, ICT and Future Planning (the first author: NRF-2018R1D1A1B07046508, the second author: NRF-2017R1C1B1011921, and the third author:  NRF-2017R1C1B1005166) .}
\begin{abstract}
This note provides a connection between the Hodge spectrum and the Hodge ideals recently developed by Musta\c{t}\v{a}--Popa in the case of weighted homogeneous isolated singularities. 
\end{abstract}
\maketitle

\section{Introduction}\label{intro}
The Hodge spectrum of a hypersurface singularity given by a local equation $f$ is a fractional Laurent polynomial
\[
Sp_f(t):=\sum_{\alpha \in \mathbb{Q}} n_{f,\alpha}t^\alpha.
\]
It encodes information about the Hodge filtration and the monodromy action on the cohomology of the Milnor fibre. If $n_{f,\alpha}$ is non-zero, then $\alpha$ is called a \emph{spectral number} and $n_{f,\alpha}$ is called the \emph{multiplicity} of $\alpha$. Budur~\cite{MR2015069} provides an interesting connection between these numbers for $0<\alpha\leq 1$ and  \emph{multiplier ideals}. However,  multiplier ideals do not have enough information about  the spectral numbers and the multiplicities for $\alpha >1$.

On the other hand, in \cite{2016arXiv160508088M, 2018arXiv180701932M, 2018arXiv180701935M} Musta\c{t}\v{a}--Popa has developed the theory of Hodge ideals motivated by Saito's theory of mixed Hodge modules \cite{MR1047415}. The Hodge ideals $I_k(D)$ of a $\Q$-divisor $D$ are indexed by non-negative integer $k$. For $k=0$, the Hodge ideal $I_0(D)$ is the multiplier ideal $\mathcal{I}((1-\epsilon)D)$ for $0<\epsilon \ll 1$. For a weighted homogeneous isolated singularity given by $f$, very recently Zhang~\cite{2018arXiv181006656Z} gives an explicit formula of Hodge ideals $I_k(\alpha Z)$ where $Z=(f=0)$ is a reduced divisor (cf. Saito \cite{MR2567401} for $\alpha=1$).

The aim of this note is to provide a relation between the whole Hodge spectrum and  the Hodge ideals in the case of weighted homogeneous polynomials with isolated singularities. The main result is Theorem~\ref{theorem}, which shows that the spectral numbers and the multiplicities appear as ``jumping numbers" and ``jumping dimension" of a decreasing sequence of ideals related to the Hodge ideals. Furthermore, we give observations on a possible generalization to other cases.

Let $f\colon (\mathbb{C}^n,0)\to(\mathbb{C},0)$ be a weighted homogeneous holomorphic function with an isolated singularity at $0\in X:=\mathbb{C}^n$ and let $Z$ denote the reduced divisor defined by $f=0$. 
For $k\in \mathbb{Z}_{\geq 0}$ and $\alpha \in (0,1] \cap \mathbb{Q}$, 
define ideals in $\mathbb{C}\{x_1,\cdots, x_n\}$ 
\[
J^{\mu}_{k+\alpha} := I_k(\alpha Z)+ \langle \partial f \rangle \text{ and } J^{\tau}_{k+\alpha} := I_k(\alpha Z)+ \langle f, \partial f \rangle
\]
where $I_k (\alpha Z)$ is the $k$-th Hodge ideal of $\alpha Z$ and $\langle \partial f \rangle$ denotes the Jacobian ideal of $f$. Since $f$ is weighted homogeneous, we have $J^\mu_{k+\alpha}=J^\tau_{k+\alpha}$ and we denote them by $J_{k+\alpha}$. Also, we have
\[
J_{\beta} \supset J_{\beta'} \quad \text{for $\beta \leq \beta'$}.
\]
Thus we can define the fractional Laurent polynomial
\begin{equation}
Sp_f^{\tau}(t):=\sum_{\beta \in \mathbb{Q}} (\dim_\mathbb{C} J_\beta / J_{>\beta})\cdot t^\beta
\end{equation}
where  $J_{>\beta}:= J_{\epsilon+ \beta}$ for $0<\epsilon \ll 1$.
The main result in this note is that this Laurent polynomial coincides with the Hodge spectrum of the singularity given by $f$.
\begin{theorem}\label{theorem}
Let $f\colon (\mathbb{C}^n,0)\to(\mathbb{C},0)$ be a weighted homogeneous holomorphic function with an isolated singularity at the origin. Let $Z=\div(f)$. For $k\in \mathbb{Z}$ and $\alpha \in (0,1] \cap \mathbb{Q}$, define 
\[
J_{k+\alpha} := I_k(\alpha Z)+ \langle \partial f \rangle
\]
to be an ideal in $\mathbb{C}\{x_1,\cdots, x_n\}$ where $I_k (\alpha Z)$ is the $k$-th Hodge ideal of $\alpha Z$ and $\langle \partial f \rangle$ is the Jacobian ideal of $f$.

Let $Sp_f(t)=\sum_{\beta\in \mathbb{Q}} n_{f, \beta} t^{\beta}$ be the Hodge spectrum of $f$. Then
	\begin{equation*}
		n_{f,\beta}=\dim_\mathbb{C} J_\beta / J_{>\beta},
	\end{equation*} 
	where  $J_{>\beta}:= J_{\epsilon+ \beta}$ for $0<\epsilon \ll 1$.
\end{theorem}
We remark that Zhang~\cite{2018arXiv181006656Z} Corollary 3.7 shows that the jumping numbers in $J_{\beta}$ coincide with the roots of the \emph{microlocal $b$-fuction $\tilde{b}_f(s)$} of a weighted homogeneous isolated singularity given by $f$.\footnote{This result also gives a description of the spectral numbers in terms of the jumping numbers as it is known that the set of the spectral numbers is equal to the set of the roots of the microlocal $b$-fuction $\tilde{b}_f(s)$ in the case of weighted homogeneous isolated singularities.
We thank Mihnea Popa for pointing this out to us.} 
 
Unfortunately, the fractional Laurent polynomial $Sp_f^{\tau}(t)$ does not coincide with the Hodge spectrum in general even for isolated singularities. This is because $Sp_f^{\tau}(1)$ would give the \emph{Tjurina number} which is the dimension of $\mathcal{O}/\langle f, \partial f \rangle$ while $Sp_f(1)$ is equal to the \emph{Milnor number} which is the dimension of $\mathcal{O}/\langle \partial f \rangle$.

\subsubsection*{\bf Layout}
Section~\ref{Sec.Preliminaries} contains the basic definitions we need from the theory of Hodge ideals and Hodge spectrums. In Section~\ref{Weighted homogeneous cases}, we focus on the weighted homogeneous isolated cases and recall known results on Hodge ideals and Hodge spectrums. We also state the main theorem and prove it. 
In Section~\ref{Sec.Non-degenerate cases}, we present meaningful examples in non-degenerate cases.
In Section~\ref{Sec.Discussion}, we discuss how to generalise this to other cases.
\subsubsection*{\bf Acknowledgement} We are deeply grateful to Nero Budur for his valuable suggestion and encouragement. His suggestion was our starting point to study this problem. We also thank Mihnea Popa to point out  the meaning of Zhang~\cite{2018arXiv181006656Z} Corollary 3.7. 
\section{Preliminaries}\label{Sec.Preliminaries}
\subsection{Hodge ideals}
Let $X$ be a smooth complex variety of dimension $n$ and let $D$ be a reduced divisor on $X$. Then there is the left $\mathcal{D}_X$-module
\[
	\mathcal{O}_X(*D)=\bigcup_{k\geq 0}\mathcal{O}_X(kD)
\]
of rational functions with poles along $D$, that is, the localization of $\mathcal{O}_X$ along $X$. This is also a left $\mathcal{D}_X$-module underlying the mixed Hodge module $j_{*}\Q^H_U[n]$, where $U=X\setminus D$ and $j\colon U\hookrightarrow X$ is the inclusion map. By \cite{MR1047415} it has the Hodge filtration $F_{k}\mathcal{O}_X(*D)$ which is contained in the pole order filtration, namely
\[
	F_k\mathcal{O}_X(*D)\subseteq\mathcal{O}_X((k+1)D) \textrm{ for all } k\geq 0.
\]
The inclusion above leads to the definition of a coherent sheaf of ideals $I_k(D)$ for each $k \geq 0$ by the formula,
\[
	F_k\mathcal{O}_X(*D)=I_k(D)\otimes\mathcal{O}_X((k+1)D).
\]
Then the ideal sheaf $I_k(D)\subseteq \mathcal{O}_X$, for $k\geq 0$, is called the \emph{$k$-th Hodge ideal}. 

Musta\c{t}\v{a}--Popa generalised this constrcution to $\Q$-divisors in~\cite{2018arXiv180701932M}. Let $D$ be an effective $\Q$-divisor on $X$. Then we can write locally $D=\alpha Z$ where $\alpha$ is a positive rational number and $Z$ is defined by a nonzero regular function $f$, that is, $Z=\div(f)$. Then the twisted version of the localization $\mathcal{D}_X$-module
\[
	\mathcal{M}(f^{-\alpha})\coloneqq \mathcal{O}_X(*Z)f^{-\alpha}
\]
has a filtration $F_k \mathcal{M}(f^{-\alpha})$, with $k\geq 0$ such that $F_k \mathcal{M}(f^{-\alpha})\subseteq \mathcal{O}_X(kZ)f^{-\alpha}$. From this, we can define the \emph{$k$-th Hodge ideal} as follows:
\[
	F_k\mathcal{M}(f^{1-\alpha})=I_k(D)\otimes_{\mathcal{O}_X}\mathcal{O}_X(kZ)f^{-\alpha}.
\]

\subsection{Milnor fibers and Hodge spectrums}
Let $f \colon (\mathbb{C}^n,0) \rightarrow (\mathbb{C},0)$ be the germ of a non-zero holomorphic function. The {\it Milnor fiber} at the origin is 
$$M_{f,0}=\{z\in \mathbb{C}^n \ | \ |z|<\epsilon \;\text{ and}\; f(z)=\delta\} \text{ for } 0<|\delta|\ll\epsilon\ll 1.$$ 
The cohomology groups $H^*(M_f,\mathbb{C})$ carry canonical mixed Hodge structures such that the semi-simple part $T_s$ of the monodromy $T$ acts as an automorphism of finite order of these mixed Hodge structures (see \cite{MR2393625}-12.1.3). By the monodromy theorem, the eigenvalues $\lambda$ of the monodromy action on $H^*(M_f,\mathbb{C})$ are roots of unity. Thus we can define the {\it spectrum multiplicity} of $f$ at $\alpha\in \mathbb{Q}$ to be 
\begin{align*}
n_{f,\alpha}&=\sum_{j\in \mathbb{Z}} (-1)^{j-n+1}\dim Gr_F^p \tilde{H}^{j}(M_{f,0},\mathbb{C})_\lambda\\
&\text{with $p=\lfloor n-\alpha \rfloor$, $\lambda=\exp(-2\pi i\alpha )$,}
\end{align*}
where $\tilde{H}^j(M_{f,0},\mathbb{C})_\lambda$ is the $\lambda$-eigenspace of the reduced cohomology under $T_s$ and $F$ is the Hodge filtration. The $\alpha$ is called a {\it spectral number} if $n_{f,\alpha}\neq 0$. The {\it Hodge spectrum} of $f$ is the fractional Laurent polynomial
$$Sp_f(t):=\sum_{\alpha \in \mathbb{Q}} n_{f,\alpha}t^\alpha,$$
which is an invariant of the singularity of analytic function germs. The Hodge spectrum $Sp_f(t)$ contains all the information of Hodge filtration $F$ and the semi-simple part $T_s$ of the momodromy if $f$ has an isolated singularity at $0\in \mathbb{C}^n$.

The \emph{Milnor algebra} (or \emph{Jacobian algebra}) $\mathcal{M}_f$ 
\[
\mathcal{M}_f:= \mathbb{C}\{x_1,\cdots,x_n\}/ \langle \partial f \rangle
\]
where $\langle \partial f \rangle$ denotes the Jacobian ideal of $f$, i.e. the ideal generated by $\{\frac{\partial f}{\partial x_i}| i=1,\ldots, n\}$. The complex dimension of $\mathcal{M}_f$ is finite if $f$ has an isolated singularity at $0\in \mathbb{C}^n$. It is called the \emph{Milnor number} of $f$, i.e.
\[ 
\mu_f:=\dim _\mathbb{C} \mathcal{M}_f.
\]
Similarly  the \emph{Tjurina algebra} $\mathcal{T}_f:=\mathbb{C}\{x_1,\cdots,x_n\}/ \langle f, \partial f \rangle$ defines the \emph{Tjurina number} $\tau_f:=\dim _\mathbb{C} \mathcal{T}_f$ for isolated singularity cases. 
Obviously, we have $\tau_f\leq \mu_f$. Note that $\tau_f= \mu_f$ if and only if $f$ is weighted homogeneous after a suitable analytic change of variables, i.e. $f \in \langle \partial f \rangle$.

For an isolated singularity, Milnor proved that $\mu_f=\dim\tilde{H}^{n-1}(M_{f,0},\mathbb{C})$ thus we have $\mu_f=Sp_f(1)$ by
the definition of the Hodge spectrum. 

	
\section{Weighted homogeneous cases}\label{Weighted homogeneous cases}
	Let $f\colon (\mathbb{C}^n,0)\to (\mathbb{C},0)$ be a weighted homogeneous polynomial function of weight $(w_1, \cdots, w_n)\in \mathbb{Q}^{n}$, that is, 
	\[
	f (\lambda^{w_1} x_1,\lambda^{w_2} x_2,\ldots, \lambda^{w_n} x_n)=\lambda f(x_1,x_2,\ldots, x_n) \text{ for all $\lambda\in \mathbb{C}$}.
	\]
	For a series $g=\sum g_m x^{m}\in \mathcal{O}=\C\{x_1,\ldots,x_n\}$, define
\[
\supp(g)=\{m\in \Z^n_{\geq 0} \mid g_m\neq 0\}.
\]
Define the weight of a monomial $x_1^{a_1}x_2^{a_2}\cdots x_n^{a_n}$ in the polynomial ring to be 
\begin{equation}\label{eq:wts of monomials}
\wt(x_1^{a_1}x_2^{a_2}\cdots x_n^{a_n})=\sum w_i + \sum a_i w_i.
\end{equation}
This weight function induces a filtration on $\mathcal{O}$ defined by
\[
\mathcal{O}^{\geq \beta}:=\{g\in \mathcal{O} \mid \wt(\supp g) \geq \beta\}.
\]
We can take a monomial basis $B$ of $\mathcal{M}_f$. 
Define
\[
B^{\geq \beta}:=\{v\in B \mid \wt(v)\geq \beta\} \quad \text{and} \quad B^{\beta}:=\{v\in B \mid \wt(v)= \beta\}.
\]
Note that $B^{\geq \beta}=\cup_{\gamma \geq \beta} B^{\gamma}$.

The ring $\mathcal{O}$ has a natural grading with respect to~(\ref{eq:wts of monomials}). For homogeneous elements $g,h\in S$, we have 
\[
\wt(gh)=\wt(g)+\wt(h) -w,
\]
where $w:=\sum_{i=1}^{n} w_i$.

Let $f_1,\cdots,f_c$ be a regular sequence of homogeneous elements in $S$. Let $d_i$ be the weight of $f_i$. Define $S^0:=\mathcal{O}$ and $S^{i}:=S^{i-1}/(f_i)$. Consider the Hilbert-Poincar\'{e} series
\[
{P}_i(t):=\sum_{\alpha} (\dim_{\C} S^i_\alpha)\cdot t^{\alpha}
\]
where $S^i_{\alpha}$ is the vector space of weight-$\alpha$ homogeneous elements in $S^i$.

First note that 
\[
{P}_0(t)= \frac{t^{w}}{\prod_{j=1}^n(1-t^{w_j})}.
\]
From the following short exact sequence
\[
0\to (f_i) \to S^{i-1} \to S^{i} \to 0,
\]
we have
\[
t^{d_i} {P}_{i-1}(t)+t^{w}{P}_{i}(t)=t^{w}{P}_{i-1}(t).
\]
By induction, we get
\begin{equation}\label{eq.HP series}
{P}_i(t) = \frac{t^w}{\prod_{j=1}^n(1-t^{w_j})} \cdot {\prod_{j=1}^i \frac{t^w-t^{d_j}}{t^w}}.
\end{equation}

\subsection{The Hodge spectrum} This section recalls some relevant formulae of the Hodge spectrum for a weighted homogeneous isolated singularity.

There is a well-known formula of the Hodge spectrum in the case of a weighted homogeneous polynomial with an isolated singularity (e.g.~\cite{MR1621831}~II-(8.4.10)).
		\begin{proposition}
			Let $f\colon(\mathbb{C}^n,0)\to(\mathbb{C},0)$ be a weighted homogeneous holomorphic function germ with an isolated singularity at the origin. Assume that the weight of $f$ is $(w_1, \cdots, w_n)$. Then
			\begin{equation}\label{spec_whi}
				Sp_f(t)=\prod_{j=1}^n \frac{t^{w_j}-t}{1-t^{w_j}}.
			\end{equation}
		\end{proposition}

Note that $\frac{\partial f}{\partial x_1},\ldots,\frac{\partial f}{\partial x_n}$ form a regular sequence of homogeneous elements in $\C\{x_1,x_2,\ldots, x_n\}$. The weight of $\frac{\partial f}{\partial x_i}$ is $1-w_i+w$. From (\ref{eq.HP series}), it follows that the Hilbert-Poincar\'{e} series $P_f(t)$ of the Milnor algebra is
	\begin{align*}
{P_f}(t)&=\frac{t^w}{\prod_{j=1}^n(1-t^{w_j})} \cdot {\prod_{j=1}^n \frac{t^w-t^{1-w_j+w}}{t^w}}\\
&=\prod_{j=1}^n \frac{t^{w_j}-t}{1-t^{w_j}} =Sp_f(t).
	\end{align*}
We have seen that:
\begin{proposition}\label{Prop.HP series and the spectrum}
Let $f$ be a weighted homogeneous polynomial of weight $(w_1, \cdots, w_n)$ with an isolated singularity at the origin. Then the Hilbert-Poincar\'{e} series ${P_f}(t)$ of the Milnor algebra coincides with the Hodge spectrum of $f$.
\end{proposition}
		
\subsection{The Hodge ideals} This section recalls the theory of the Hodge ideals for weighted homogeneous isolated singularities.

%
%
\begin{theorem}[\cite{2018arXiv181006656Z}]
	If $D = \alpha Z$, where $0 < \alpha \leq 1$ and $Z$ is a reduced effective divisor defined by $f$, a weighted homogeneous polynomial with an isolated singularity at the origin, then we have
	\[
		F_k\mathcal{M}(f^{1-\alpha})=\sum_{i=0}^k F_{k-i}\mathcal{D}_X\cdot \left(\frac{\mathcal{O}^{\geq \alpha+i}}{f^{i+1}}f^{1-\alpha}\right),
	\]
	where the action $\cdot$ of $\mathcal{D}_X$ on the right hand side is the action on the left $\mathcal{D}_X$-module $\mathcal{M}(f^{1-\alpha})$ defined by
	\[
		D\cdot (wf^{1-\alpha})\coloneqq \left(D(w)+w\frac{(1-\alpha)D(f)}{f}\right)f^{1-\alpha}, \textrm{ for any } D\in Der_{\C}\mathcal{O}_X
	\]
\end{theorem}

\begin{corollary}[\cite{2018arXiv181006656Z}]\label{Cor.[Zhang]}
	If $D = \alpha Z$, where $0 < \alpha \leq 1$ and $Z$ is a reduced effective divisor defined by $f$, a weighted homogeneous polynomial with an isolated singularity at the origin, then we have
	\[
		I_0(D)=\mathcal{O}^{\geq \alpha}
	\]
	and
	\[
		I_{k+1}(D)=\sum_{v_j\in B^{\geq k+1+\alpha}} \C v_j + \sum_{\substack{1\leq i\leq n \\ a\in I_k(D)}} \langle f\partial_i a-(\alpha+k)a\partial_i f \rangle.
	\]
\end{corollary}
\subsection{Main Theorem}
For $k\in \Z_{\geq 0}$ and $\alpha\in \Q\cap(0,1]$, define 
\[
J_{k+\alpha}:=I_{k}(\alpha Z) + \langle \partial f \rangle.
\]
to be an ideal in $\mathbb{C}\{x_1,\cdots, x_n\}$ where $I_k (\alpha Z)$ is the $k$-th Hodge ideal of $\alpha Z$ for $Z=\div(f)$ and $\langle \partial f \rangle$ is the Jacobian ideal of $f$. 
By~Corollary~\ref{Cor.[Zhang]}, we know that $(J_{\beta})_{\beta\in \Q}$ is a decreasing sequence. Thus we can define the fractional Laurent polynomial
\[
Sp_f^{\tau}(t):=\sum_{\beta \in \mathbb{Q}} (\dim_\mathbb{C} J_\beta / J_{>\beta})\cdot t^\beta,
\]
where  $J_{>\beta}:= J_{\epsilon+ \beta}$ for $0<\epsilon \ll 1$.

Let $Sp_f(t)$ be the Hodge spectrum of $f$. Our main result, which is equivalent to Theorem~\ref{theorem}, is:
\begin{theorem}[Main theorem]\label{main theorem} For a weighted homogeneous isolated singularity defined by~$f$, we have
\[
Sp_f(t)=Sp_f^{\tau}(t).
\]
\end{theorem}
\begin{proof}
Proposition~\ref{Prop.HP series and the spectrum} shows that the Hilbert-Poincar\'{e} series ${P_f}(t)$ of the Milnor algebra coincides with the Hodge spectrum of $f$. Note that
\[
{P_f}(t)=\sum_{\beta \in \mathbb{Q}} |B^{\beta}|\cdot t^\beta
\]
where $|B^{\beta}|$ denotes the number of elements in $B^{\beta}$.

On the other hand, from Corollary~\ref{Cor.[Zhang]}, it follows that
\[
J_{\beta}=\langle B^{\geq \beta} \rangle + \langle \partial f \rangle.
\]
Since $B \cap \langle \partial f \rangle=\emptyset$, we have
\[
\dim_{\C} J_\beta / J_{>\beta}=|B^{\beta}|.
\]
This implies that 
\[
Sp_f(t)={P_f}(t)=Sp_f^{\tau}(t).\qedhere
\]
\end{proof}

\section{Non-degenerate cases}\label{Sec.Non-degenerate cases}
In this section we distinguish two ideals $J^{\mu}_{\beta}$ and $J^{\tau}_{\beta}$ defined in Section \ref{intro}.
Consider two fractional Laurent polynomials 
\[
Sp_f^{\mu}(t):=\sum_{\beta \in \mathbb{Q}} (\dim_\mathbb{C} J^{\mu}_\beta / J^{\mu}_{>\beta})\cdot t^\beta \text{ and }Sp_f^{\tau}(t):=\sum_{\beta \in \mathbb{Q}} (\dim_\mathbb{C} J^{\tau}_\beta / J^{\tau}_{>\beta})\cdot t^\beta.
\] 

Let $f=\sum f_m x^{m}\in \mathcal{O}=\C\{x_1,\cdots,x_n\}$ be a germ of a hypersurface singularity. For a series $g=\sum g_m x^{m}$, define
\[
\supp(g)=\{m\in \Z^n_{\geq 0} \mid g_m\neq 0\}.
\]
The \emph{Newton boundary} $\Gamma(f)$ of $f$ is the set of compact faces (not contained in any hyperplane $x_i=0$) of the Newton polyhedron of $f$. We say that $f$ is \emph{non-degenerate} if no common zeros of $\partial_i f_\sigma$ are in $(\C^{\times})^{n}$ for every $\sigma \in \Gamma(f)$ where $f_\sigma=\sum_{m\in \Cone(\sigma)} f_m x^m$. 
We may assume that $f$ contains $x_i^{m_i}$ for all~$i$. This implies that 
\[
\cup_{\sigma \in \Gamma(f)} \Cone(\sigma)=\R^n_{\geq 0}.
\]

We can define a piecewise linear function $h\colon \R_{\geq 0}^n \to \R_{\geq 0}$ satisfying:
\begin{enumerate}
\item $h$ is linear on $\Cone(\sigma)$ for all $\sigma \in \Gamma(f)$, and
\item $h\mid_{\sigma}=1$ for all $\sigma \in \Gamma(f)$.
\end{enumerate}
The function $h$ can be constructed as follows. For each codimension one face $\sigma \in \Gamma(f)$, there is a linear form $l_{\sigma}$ such that $l_{\sigma}\equiv 1$ on $\sigma$. For $m\in \R^n_{\geq 0}$, define
\[
h(m):=\inf \{l_{\sigma}(m) \mid  \text{$\sigma$ is a codimension one face}\}.
\]
This defines a decreasing filtration on $\mathcal{O}$ by
\[
\mathcal{F}^{\beta} \mathcal{O}=\{g\in \mathcal{O} \mid h(\supp g) \geq \beta\},
\]
which is called the \emph{Newton filtration}. The Newton filtration induces a weight function ${\rho} \colon \mathcal{O}\to \Q$ defined by
\[
{\rho} (g) = \max \{\beta \mid g\in \mathcal{F}^{\beta}\mathcal{O}\}.
\]
Note that this weight function descends to the Milnor algebra $\mathcal{M}_f$ as follows:
\[
\bar{\rho}([g])=\max \{\rho(g')\mid g'\equiv g \bmod \langle \partial f \rangle \}.
\]
Define
\[
{\mathcal{O}}^{\geq \beta}=\{g \in \mathcal{O} \mid {\rho}(g\cdot x_1\cdots x_n )\geq \beta\}.
\]

\begin{remark}
For $g\in \mathcal{O}^{\geq \beta}$, we have that $gf\in \mathcal{O}^{\geq \beta+1}$.
\end{remark}

For a non-degenerate isolated singularity, using Zhang's idea in~\cite{2018arXiv181006656Z}, we can find that (in the same way):
\begin{equation}\label{eq of Hodge ideals for non-degenerate case}
I_k(\alpha Z)={\mathcal{O}}^{\geq \alpha+k}+\sum_{\substack{1\leq i\leq n \\ a\in I_{k-1}(\alpha Z)}} \langle f \partial_i a -(\alpha+k-1)a\partial_i f \rangle.
\end{equation}
This implies that
\begin{equation}\label{eq of J tau for non-degenerate case}
J^{\tau}_{k+\alpha}:=I_k(\alpha Z)+\langle f, \partial f \rangle ={\mathcal{O}}^{\geq \alpha+k}+\langle f, \partial f \rangle,
\end{equation}
for $k\in \Z_{\geq 0 }$ and $\alpha\in \Q\cap(0,1]$. Thus
\[
\dim_{\C} J^{\tau}_{\beta}/J^{\tau}_{>\beta}
\]
is the dimension of the vector space spanned by the homogeneous elements of weight $\beta$ not in the ideal $\langle f, \partial f \rangle$. We may consider the fractional Laurent polynomial
\[
Sp_f^{\tau}(t):=\sum_{\beta \in \mathbb{Q}} (\dim_\mathbb{C} J^{\tau}_\beta / J^{\tau}_{>\beta})\cdot t^\beta
\]
but this should not coincide with the Hodge spectrum in general. Indeed, we have $Sp_f^{\tau}(1)\leq \tau_f$ where $\tau_f$ is the Tjurina number of $f$, which is less than the Milnor number in general.
\begin{example}\label{ex.f=x^r+x2y2+y^s}
Let $f:=x^r+x^2y^2+y^s$ with $s\geq r\geq 4$ and $s > 4$. Note that the weight function $h$ is defined by
\[
h(x^a y^b):=
\begin{cases}
\frac{b}{2}+\frac{a-b}{r} &\text{if $a \geq b$,}\\
\frac{a}{2}+\frac{b-a}{s} &\text{if $a \leq b$.}\\
\end{cases}
\]
In this case, $g \in {\mathcal{O}}^{\geq \beta}$ if and only if $h(g)\geq \beta-\frac{1}{2}$, i.e. we have
\[
{\rho}(g \cdot xy)={\rho}(g)+ \frac{1}{2}.
\]
Furthermore, note that ${\mathcal{O}}^{\geq \frac{3}{2}}\subset \langle f, \partial f \rangle$.\footnote{It is because $\mathcal{O}^{\geq \frac{3}{2}}$ is generated by $x^r,x^2y^2,xy^{\lceil\frac{s+2}{2}\rceil},x^{\lceil\frac{r+2}{2}\rceil}y,y^r$. However, this does not hold in general, e.g. $f=x^7+y^7$ we have $\mathcal{O}^{\geq \frac{9}{7}}\not\subset \langle f, \partial f \rangle$.}

Let $D=\alpha Z$ with $\alpha\in (0,1]\cap \Q$. Then 
\[
I_0(D)={\mathcal{O}}^{\geq \alpha}.
\]
From (\ref{eq of Hodge ideals for non-degenerate case}), we have
\[
I_1(D)={\mathcal{O}}^{\geq \alpha+1}+\sum_{\substack{1\leq i\leq 2 \\ a\in I_{0}(D)}} \langle f \partial_i a -a\alpha\partial_i f \rangle.
\]
For $\alpha \leq \frac{1}{2}$, we have $I_0(\alpha Z)=\mathcal{O}$ and 
\[
I_1(\alpha Z) +\langle \partial f \rangle= {\mathcal{O}}^{\geq \alpha+1} +\langle \partial f \rangle
\]
as $f \in {\mathcal{O}}^{\geq \frac{3}{2}}$.
Then for $\alpha=\frac{1}{2}+\epsilon$ with $0<\epsilon\ll 1$, 
\begin{align*}
I_0((\frac{1}{2}+\epsilon) Z)&=\langle x,y \rangle,\\
I_1((\frac{1}{2}+\epsilon) Z)+\langle \partial f \rangle&= {\mathcal{O}}^{\geq \alpha+1}+ \langle f \rangle+\langle \partial f \rangle=\langle f, \partial f \rangle
\end{align*} 
Thus we have 
\[
J^{\tau}_{\beta}=J^{\mu}_{\beta}\text{ for $\beta \leq \frac{3}{2}+\epsilon$}.
\]
This means that
\[
Sp_f^{\tau}(t)=\sum_{\beta \in \mathbb{Q}} (\dim_\mathbb{C} J^{\tau}_\beta / J^{\tau}_{>\beta})\cdot t^\beta =  \sum_{\beta \in \mathbb{Q}\cap (0,\frac{3}{2}]} (\dim_\mathbb{C} J^{\mu}_\beta / J^{\mu}_{>\beta})\cdot t^\beta
\]
should not coincide with the Hodge spectrum as 
\[
Sp_f^{\tau}(1)=\tau_f=\mu_f-1.
\] 
Indeed, since the Hodge spectrum of $f$ is \[
Sp_f(t)=t^\frac{1}{2}\left(t^{\frac{1}{r}}+\cdots +t^{\frac{r-1}{r}} \right)+t^\frac{1}{2}\left(t^{\frac{1}{s}}+\cdots +t^{\frac{s-1}{s}}\right)+\left( t^{\frac{1}{2}}+t+t^{\frac{3}{2}}\right)
\]
by \cite{MR1621831}~II-(8.5.13), we have
\[
Sp_f(t)=Sp_f^{\tau}(t)+t^{\frac{3}{2}}.
\]
\end{example}
\begin{example}\label{ex. x5+x2y2+y5}
Let $f=x^5+x^2y^2+y^5$. 
Note that we have
\[
1\in \mathcal{O}^{\geq\frac{1}{2}},
f\in \mathcal{O}^{\geq \frac{3}{2}},\text{ and }
\mathcal{O}^{\geq \frac{3}{2}}\subset \langle f, \partial f \rangle.\footnote{$\mathcal{O}^{\geq \frac{3}{2}}= \langle x^5,x^4y,x^2y^2,xy^4,y^5 \rangle $.}
\]
For $0<\alpha\leq \frac{1}{2}$, we have
\begin{align*}
I_0(\alpha Z)&={\mathcal{O}}^{\geq \alpha}=\mathcal{O},\\
I_1(\alpha Z)&={\mathcal{O}}^{\geq \alpha+1}+\langle \partial f \rangle.
\end{align*}
For $\frac{1}{2}<\alpha\leq \frac{7}{10}$,
\begin{align*}
I_0(\alpha Z)&={\mathcal{O}}^{\geq \alpha}= \langle x,y \rangle ,\\
I_1(\alpha Z)+\langle \partial f \rangle&=
{\mathcal{O}}^{\geq \alpha+1}+\sum_{\substack{1\leq i\leq 2 \\ a\in \langle x,y \rangle }} \langle f \partial_i a -a\alpha\partial_i f \rangle +\langle \partial f \rangle\\
&= {\mathcal{O}}^{\geq \alpha+1} + \langle f \rangle +\langle \partial f \rangle=\langle f, \partial f \rangle.
\end{align*} 
For $\frac{7}{10}<\alpha\leq \frac{9}{10}$,
\begin{align*}
I_0(\alpha Z)&={\mathcal{O}}^{\geq \alpha}= \langle x^2,xy, y^2 \rangle \\
I_1(\alpha Z)+\langle \partial f \rangle&=
{\mathcal{O}}^{\geq \alpha+1}+\sum_{\substack{1\leq i\leq 2 \\ a\in \langle x^2,xy,y^2 \rangle }} \langle f \partial_i a -a\alpha\partial_i f \rangle +\langle \partial f \rangle\\
&= {\mathcal{O}}^{\geq \alpha+1} + \langle xf,yf \rangle +\langle \partial f \rangle= \langle xf,yf, \partial f \rangle =\langle \partial f \rangle.
\end{align*} 
For $\frac{9}{10}<\alpha\leq 1$
\begin{align*}
I_0(\alpha Z)&={\mathcal{O}}^{\geq \alpha}= \langle x^3,xy, y^3 \rangle \text{ and }\\
I_1(\alpha Z)+\langle \partial f \rangle&=
{\mathcal{O}}^{\geq \alpha+1}+\sum_{\substack{1\leq i\leq 2 \\ a\in \langle x^3,xy, y^3 \rangle }} \langle f \partial_i a -a\alpha\partial_i f \rangle+\langle \partial f \rangle\\
&= {\mathcal{O}}^{\geq \alpha+1} + \langle xf,yf \rangle+\langle \partial f \rangle= \langle xf,yf, \partial f \rangle =\langle \partial f \rangle.
\end{align*} 
Here, we use 
\[
\mathcal{O}^{\geq \frac{18}{10}}=\langle x^7,x^5y,x^4y^2,x^3y^3,x^2y^4,xy^5,y^7 \rangle \subset \langle xf,yf, \partial f \rangle.
\]
Note that $\langle xf,yf, \partial f \rangle=\langle \partial f \rangle$ as ideals in $\mathcal{O}=\C\{x,y\}$.\footnote{Warning: $\langle xf,yf, \partial f \rangle \neq\langle \partial f \rangle$ as ideals in the polynomial ring $\C[x,y]$.}
From the calculation above, it follows that
\begin{equation*}
J^{\mu}_\beta / J^{\mu}_{>\beta}=
\begin{cases}
\C \cdot 1&\text{if }\beta=\frac{1}{2},\\
\C \cdot x\oplus \C \cdot y &\text{if }\beta=\frac{7}{10},\\
\C \cdot x^2\oplus \C \cdot y^2 &\text{if }\beta=\frac{9}{10},\\
\C \cdot xy &\text{if }\beta=1,\\
\C \cdot x^3\oplus \C \cdot y^3 &\text{if }\beta=\frac{9}{10},\\
\C \cdot x^4\oplus \C \cdot y^4 &\text{if }\beta=\frac{9}{10},\\
\C \cdot y^5 &\text{if }\beta=\frac{17}{10},\\
0 &\text{otherwise}.
\end{cases}
\end{equation*}
This implies that
\begin{align*}
Sp_f(t)&=t^\frac{1}{2}+2 t^\frac{7}{10}+2 t^\frac{9}{10}+ t+2 t^\frac{11}{10}+2 t^\frac{13}{10}+t^\frac{3}{2}\\
Sp^{\tau}_f(t)&=t^\frac{1}{2}+2 t^\frac{7}{10}+2 t^\frac{9}{10}+ t+2 t^\frac{11}{10}+2 t^\frac{13}{10}\\
Sp^{\mu}_f(t)&=t^\frac{1}{2}+2 t^\frac{7}{10}+2 t^\frac{9}{10}+ t+2 t^\frac{11}{10}+2 t^\frac{13}{10}+t^\frac{17}{10}
\end{align*}
Note that $Sp^{\mu}_f(1)=Sp_f(1)=\mu_f$ but $Sp^{\mu}_f(t)\neq Sp_f(t)$. 

\end{example}
\section{Discussion}\label{Sec.Discussion}
In~\cite{2018arXiv180701935M}, Musta\c{t}\v{a}--Popa proved that
\begin{equation}\label{eq:V-filtration}
I_{k}(\alpha Z) =\tilde{V}^{\alpha+k}\mathcal{O} \mod \langle f \rangle,
\end{equation}
where $\tilde{V}^{\bullet}\mathcal{O}$ means the \emph{microlocal $V$-filtration} on $\mathcal{O}$ induced by $f$.

By (\ref{eq:V-filtration}), we know that $(J^{\tau}_{\beta})_{\beta\in \Q_{\geq 0}}$ satisfies that
\[
J^{\tau}_{\beta} \subset J^{\tau}_{\beta'}\quad \text{if $\beta \geq \beta'$}.
\]
Define the fractional Laurent polynomial
\[
Sp_f^{\tau}(t):=\sum_{\beta \in \mathbb{Q}} (\dim_\mathbb{C} J^{\tau}_\beta / J^{\tau}_{>\beta})\cdot t^\beta.
\]
\begin{question}\label{Question:Tjurina number}
For $f$ having an isolated singularity, do we have
\[
Sp_f^{\tau}(1)=\sum_{\beta \in \mathbb{Q}} (\dim_\mathbb{C} J^{\tau}_\beta / J^{\tau}_{>\beta})=\tau_f,
\]
where $\tau_f$ is the Tjurina number of $f$?
\end{question}
\begin{remark}
Question~\ref{Question:Tjurina number} is affirmative if $J_{\beta}=\langle f, \partial f \rangle$ for $\beta\gg 0$.
\end{remark}

Another sequence of ideals $(J^{\mu}_{\beta})_{\beta\in \Q}$
looks closely related to the Hodge spectrum of $f$. However, there is no guarantee that $(J^{\mu}_{\beta})_{\beta\in \Q}$ is a decreasing sequence. Example~\ref{ex. x5+x2y2+y5} illustrates that even if $(J^{\mu}_{\beta})_{\beta\in \Q}$ defines a decreasing sequence, the fractional Laurent polynomial
\[
Sp_f^{\mu}(t):=\sum_{\beta \in \mathbb{Q}} (\dim_\mathbb{C} J^{\mu}_\beta / J^{\mu}_{>\beta})\cdot t^\beta
\]
does not coincide with the Hodge spectrum in general.
\begin{question} Is $Sp_f^{\mu}(t)$ well-defined? In other words, is $(J^{\mu}_{\beta})_{\beta\in \Q}$ a decreasing sequence? If so, $Sp_f^{\mu}(1)=\mu_f$?
\end{question}

\bibliographystyle{alpha}
\bibliography{reference}

\end{document}